\documentclass{amsart}
\usepackage{amsthm,stmaryrd,enumerate}
\usepackage{amssymb}
\usepackage{epsfig}
\usepackage{verbatim}
\usepackage{hyperref}
\usepackage{mathrsfs}
\usepackage[all]{xy}
\usepackage{color}
\newcommand{\brk}[1]{{\left\langle{#1}\right\rangle}}

\newcommand{\ve}{\varepsilon}
\newcommand{\ro}{r}

\newcommand{\Xr}{{X_r}}

\newcommand{\srcol}{{\C\setminus \Xr}}
\newcommand{\A}{{\mathsf A}}

\newcommand{\Nr}{{\mathsf N}_\ro}
\newcommand{\PP}{{\mathsf P}}
\newcommand{\sot}{{\footnotesize \mathsf{SO(3)}}}
\newcommand{\sud}{{\footnotesize \mathsf{SU(2)}}}

\newcommand{\Hr}{H_r}
\newcommand{\qr}{{q}}
\newcommand{\coh}{\omega}
\newcommand{\vp}{\varphi}
\newcommand{\ord}{\operatorname{ord}}
\newcommand{\WRT}{\operatorname{WRT}}

\newcommand{\e}{{\operatorname{e}}}
\newcommand{\slt}{{\mathfrak{sl}(2)}}
\newcommand{\UsltH}{{U_q^{H}\slt}}
\newcommand{\Ubar}{{\wb U_q^{H}\slt}}
\newcommand{\cat}{\mathscr{C}}
\newcommand{\catd}{\mathcal{D}}

\newcommand{\Id}{\operatorname{Id}}
\newcommand{\bp}[1]{{\left(#1\right)}}
\newcommand{\qn}[1]{{\left\{#1\right\}}}
\newcommand{\qN}[1]{{\frac{\qn{#1}}{\qn1}}}
\newcommand{\qd}{{\mathsf d}}
\newcommand{\qdim}{\operatorname{qdim}}
\newcommand{\End}{\operatorname{End}}
\newcommand{\Hom}{\operatorname{Hom}}
\newcommand{\sign}{\operatorname{sign}}
\newcommand{\C}{\ensuremath{\mathbb{C}} }
\newcommand{\Z}{\ensuremath{\mathbb{Z}} }
\newcommand{\Zd}{\Z/2\Z }

\newcommand{\wt}{\widetilde}
\newcommand{\wb}{\overline}
\newcommand{\ds}{\displaystyle}
\newcommand{\sig}{{\tau}}
\newcommand{\ms}[1]{\mbox{\tiny$#1$}}
\newcommand{\epsh}[2]
         {\begin{array}{c} \hspace{-1.3mm}
        \raisebox{-4pt}{\epsfig{figure=#1,height=#2}}
        \hspace{-1.9mm}\end{array}}
        
\newcommand{\nota}[1]{\marginpar{\scriptsize #1}}

\newtheorem{Df}{Definition}
\newtheorem{definition}[Df]{Definition}
\newtheorem{theorem}[Df]{Theorem}

\newtheorem{prop}[Df]{Proposition}
\newtheorem{proposition}[Df]{Proposition}
\newtheorem{lemma}[Df]{Lemma}

\newtheorem{remark}[Df]{Remark}

\newtheorem{corollary}[Df]{Corollary}
\newtheorem{conjecture}[Df]{Conjecture}

\newcounter{exo} \newcounter{numexercice}
\renewcommand{\theexo}{\arabic{exo}} 

\begin{document}
\title[WRT and non semi-simple $\slt$ 3-manifold invariants]
{Relations between Witten-Reshetikhin-Turaev and non semi-simple
  $\slt$ 3-manifold invariants} 

\author[Costantino]{Francesco Costantino}
\address{Institut de Recherche Math\'ematique Avanc\'ee\\
  Rue Ren\'e Descartes 7\\
  67084 Strasbourg, France} \email{costanti@math.unistra.fr}

\author{Nathan Geer}
\address{Mathematics \& Statistics\\
  Utah State University \\
  Logan, Utah 84322, USA}
\thanks{The first author's research was supported by French ANR
  project ANR-08-JCJC-0114-01. Research of the second author was
  partially supported by NSF grants  DMS-1007197 and DMS-1308196.  }\
\email{nathan.geer@usu.edu}

\author{Bertrand Patureau-Mirand}
\address{Univ. Bretagne-Sud,  UMR 6205, LMBA, F-56000 Vannes, France}
\email{bertrand.patureau@univ-ubs.fr}

\begin{abstract}
  The Witten-Reshetikhin-Turaev invariants extend the Jones
  polynomials of links in $S^3$ to invariants of links in 3-manifolds.
  Similarly, in \cite{CGP}, the authors constructed two 3-manifolds
  invariant $\Nr$ and $\Nr^0$ which extend the Akutsu-Deguchi-Ohtsuki
  invariant of links in $S^3$ colored by complex numbers to links in
  arbitrary manifolds.  All these invariants are based on
  representation theory of
  the quantum group $U_qsl_2$, where the definition of the invariants
  $\Nr$ and $\Nr^0$ uses a non-standard category of $U_qsl_2$-modules
  which is not semi-simple.
 In this paper we
  study the second invariant $\Nr^0$ and consider its relationship
  with the WRT invariants.  In particular,
  we show that the ADO invariant of a knot in $S^3$ is a meromorphic
  function of its color and we provide a strong relation between its
  residues and the colored Jones polynomials of the knot. Then we
  conjecture a similar relation between $\Nr^0$ and a WRT invariant.
  We prove this conjecture when the 3-manifold $M$ is not a rational
  homology sphere and when $M$ is a rational homology sphere obtained
  by surgery on a knot in $S^3$ or when $M$ is a connected sum of such
  manifolds.
\end{abstract}

\maketitle

\section*{Introduction}

In \cite{Wi}, Witten proposes a program to construct a topological
invariant of $3$-manifolds from the viewpoint of quantum mathematical
physics.  Reshetikhin and Turaev \cite{RT} give rigorous construction
of these invariants which have become known as quantum invariants of
$3$-manifolds.  These invariants are defined via surgery presentations
of a $3$-manifold.  The best known example of these invariants is a
weighted sum of colored Jones polynomials.  The invariants of
Reshetikhin and Turaev generalize to the setting of modular
categories.  Some of the common obstructions to applying this
construction to any ribbon tensor category $\catd$ include the
following facts: (1) the simple objects may have zero ``quantum
dimension", (2) there might be infinitely many isomorphism classes of
simple objects in $\catd$ and (3) $\catd$ might be non-semi-simple.
In \cite{CGP} the authors derive a general categorical setting where
these obstructions can be overcome.  In particular, they show that the
category $\cat$ of nilpotent representations of a generalized version
of quantized $\mathfrak{sl}(2)$ at a primitive $r$\textsuperscript{th}
ordered root of unity gives rise to two invariants: $\Nr$ and $\Nr^0$.
In this paper we investigate the invariant $\Nr^0$.

Let $\cat$ be the category mentioned above and defined in Subsection
\ref{SS:QuantMod}.  This category has complex family of weight modules
divided into typical and atypical modules.  Here all the atypical
modules have integral weights.

Let $F$ be the usual Reshetikhin-Turaev invariant of links in $S^3$
arising from $\cat$.  The invariant $F$ has the following properties:
\begin{itemize}
\item If $L$ is a link whose components are all colored by simple modules of
  $\cat$ with integral weights then $F$ is determined by the Kauffman bracket
  and so is a version of the colored Jones polynomial.
\item If $L$ is a link with a component colored by a typical module
  then $F(L)=0$.
\end{itemize}
In \cite{GPT}, the second two authors and Turaev give an extension of
$F$ to links colored with modules in $\cat$ with non-integral weights
(see also \cite{CM}).  In particular, we construct an invariant $F'$
defined on links with at least one component colored by a typical
module.
$F'$ is a generalization of the links invariants defined by Akutsu, Deguchi
and Ohtsuki in \cite{ADO}.
We have the following relation
$$F'(L_1\sqcup L_2)=F'(L_1)F(L_2)$$
where $L_1$ is in the domain of $F'$ and $L_2$ is any $\cat$-colored link.  From this relation it follows that $F'$ recovers $F$:  if $L$ is any $\cat$-colored link then 
\begin{equation}\label{E:FF'}
F(L)= \frac{F'(L\sqcup o)}{F'(o)}
\end{equation} 
where $o$ is an unknot colored by any typical module.  Thus, $F'$ is a
kind of extension of the colored Jones polynomial to complex colors.
Furthermore, as we will show in Corollary \ref{cor:residuejones}, the
invariant $F'(K_{\alpha})$ of a knot $K\subset S^3$ colored by a
typical module of weight $\alpha\in \C$ is a meromorphic function of
${\alpha}$ whose residues at the integers
are proportional to the colored Jones polynomials of $K$ evaluated at
$q=\exp(\frac{i\pi}{r})$.  This relation allows us to re-prove the well
known Symmetry Principle (see \cite{KM}) for the colored Jones
polynomials of $K$ using a mainly graphical argument detailed in
Corollary \ref{C:LaurentPoly} (see Remark \ref{rem:symmetry}).

In \cite{CGP}, the authors layout a relationship between $\Nr$ and
$\Nr^0$ analogous to that outlined above between $F'$ and $F$; we will
now briefly recall this relation.  The invariants $\Nr$ and $\Nr^0$
are WRT-type 3-manifold invariant which are certain weighted sums of
$F'(L)$ where $L$ is a surgery presentation of $M$.  These invariants
are topological invariants of triples (a closed oriented 3-manifold
$M$, a link $T$ in $M$, an element $\coh$ in $H^1(M\setminus T;
\C/2\Z)$).  Here for $\Nr$ the triples must satisfy some requirements
of ``typicality'' as in the case of $F'$.  The invariant $\Nr^0$ is
zero unless $\coh$ is in the image of the natural map
$H^1(M;\Z/2\Z)\to H^1(M;\C/2\Z)$ induced by the universal coefficient
theorem. (Compare this with the above statement that $F(L)$ is zero if
at least one component of $L$ is colored by an atypical module.)
Finally, $\Nr$ recovers $\Nr^0$ (compare

 with Equation \eqref{E:FF'}): 
$$\Nr^0(M,T,\coh)=
\frac{\Nr \big((M,T,\coh)\#(M',T',\coh')\big)}{\Nr(M',T',\coh')}$$ 
where $(M',T',\coh')$ is a triple where $\Nr$ does not vanish (for
further details on the notion of connected sum see \cite{CGP}).

Since $F$ is essentially the colored Jones polynomial, the above
analogy leads us to the question: Is $\Nr^0$ related to the
WRT-invariant?  The purpose of this paper is to answer this question
positively for certain types of triples $(M,T,\coh)$.  To formulate
this properly we must define the WRT-invariant of a triple
$(M,T,\coh)$.  Kirby and Melvin \cite{KM} and Blanchet \cite{Bl}
consider a WRT-type invariants of $(M,\coh)$ where $\coh \in
H^1(M,\Zd)$.  In Theorem \ref{T:RefinedWRT} we give a slight
generalization of their invariants to triples of the form $(M,T,\coh)$
where $T$ is a link in $M$ and $\coh \in H^1(M\setminus T,\Zd)$.  We
denote this invariant by $\WRT_r(M,T,\coh)$.  The question above can
be formulated in the following conjecture.

If $G$ is a finite abelian group let $\ord(G)$ be the order of $G$,
i.e. the number of elements in the set underlying $G$.  If $G$ is an
infinite abelian group set $\ord(G)=0$.
\begin{conjecture}\label{C:MainConj}
Let $(M,T,\coh)$ be a compatible triple where $\coh$ take values in $\Z/2\Z
\subset \C/2\Z$.  Then  $$\Nr^0(M,T,\coh)=\ord(H_1(M; \Z)) \WRT_r(M,T,\coh).$$  
\end{conjecture}
Remark that if an abelian group $G$ has a square presentation matrix
$A\in{\mathcal M}_n(\Z)$ then $\ord(G)=|\det(A)|$.  In particular, if
a $3$-manifold is obtained by surgery on a link in $S^3$ whose linking
matrix is $A$ then $A$ is a presentation matrix for $H_1(M; \Z)$ thus
$\ord(H_1(M; \Z))=|\det(A)|$.

It should be noticed that the invariant $\Nr$ does not reduce to
$\Nr^0$.  For example, the invariant $WRT_r$ is trivial for $r=2$ and
$\Nr^0$ should only depends of $H_1(M,\Z)$.  But for $r=2$, the
invariant $\Nr$ is related to the Reidemeister torsion.  This is shown in
\cite{BCGP} where the two invariants $\Nr$ and $\Nr^0$ are extended to
manifolds with boundary using the setting of topological quantum field
theory.

In Sections \ref{S:ConjKnot} and \ref{S:ConjNonHomol} we prove this
conjecture in the following two cases: (1) when $M$ is an empty
rational homology sphere obtained by surgery on a knot in $S^3$ (or
more in general a connected sum of manifolds of this type) and (2)
when the first Betti number of $M$ is greater than zero.

\section{Preliminaries}
\subsection{Notation} \label{SS:NotationNr} All manifolds in the
present paper are oriented, connected and compact unless explicitly
stated.  All tangles in this paper will be framed and oriented.  Given
a set $Y$, a graph is said to be \emph{$Y$-colored} if it is equipped
with a map from the set of its edges to $Y$.

Let $r$ be an integer greater or equal to $2$ and let
$\qr=\e^{i\pi/r}$.  For $x\in \C$, we use the notation $q^x$ for
$\e^{x i\pi/r}$ and set $\qn x=q^{x}-q^{-x}$.  Let $\Xr=\Z\setminus
r\Z \subset \C$ and define the \emph{modified dimension}
$\qd:\srcol\to\C$ by
\begin{equation}\label{E:Def_qd}
  \qd(\alpha)=(-1)^{r-1}\prod_{j=1}^{r-1} \frac{\{j\}}{\{\alpha+r-j\}}
  =(-1)^{r-1}\frac{r\,\qn{\alpha}}{\qn{r\alpha}}.
\end{equation}
Finally, let
\begin{equation}
  \label{eq:Hr}
  \Hr=\{1-\ro,3-\ro,\ldots,\ro-3,\ro-1\}.
\end{equation}

\subsection{A quantization of $\slt$ and some of its modules}\label{SS:QuantMod}
Here we give a slightly generalized version of quantum $\slt$.  Let
$\UsltH$ be the $\C$-algebra given by generators $E, F, K, K^{-1}, H$
and relations:
\begin{align*}
  HK&=KH, & 
  [H,E]&=2E, & [H,F]&=-2F,\\
  KEK^{-1}&=q^2E, & KFK^{-1}&=q^{-2}F, &
  [E,F]&=\frac{K-K^{-1}}{q-q^{-1}}.
\end{align*}
The algebra $\UsltH$ is a Hopf algebra where the coproduct and counit
are defined by
\begin{align*}
  \Delta(E)&= 1\otimes E + E\otimes K, 
  &\varepsilon(E)&= 0, 
  \\
  \Delta(F)&=K^{-1} \otimes F + F\otimes 1,  
  &\varepsilon(F)&=0,
    \\
  \Delta(H)&=H\otimes 1 + 1 \otimes H, 
  & \varepsilon(H)&=0, 
  \\
  \Delta(K)&=K\otimes K,
  &\varepsilon(K)&=1.
\end{align*}
Define $\Ubar$ to be the Hopf algebra $\UsltH$ modulo the relations
$E^\ro=F^\ro=0$.

Let $V$ be a finite dimensional $\Ubar$-module.  An eigenvalue
$\lambda\in \C$ of the operator $H:V\to V$ is called a \emph{weight}
of $V$ and the associated eigenspace is called a \emph{weight space}.
We call $V$ a \emph{weight module} if $V$ splits as a direct sum of
weight spaces and $\qr^H=K$ as operators on $V$.  Let $\cat$ be the
category of finite dimensional weight $\Ubar$-modules.  The category
$\cat$ is a ribbon Ab-category, see \cite{GPT,Mu,Oh}.

We will now recall the definition of the duality morphisms and the
braiding of the category $\cat$.  Let $V$ and $W$ be objects of
$\cat$.  Let $\{v_i\}$ be a basis of $V$ and $\{v_i^*\}$ be a dual
basis of $V^*=\Hom_\C(V,\C)$.  Then
\begin{align*}
  b_{V} :& \C \rightarrow V\otimes V^{*}, \text{ given by } 1 \mapsto
  \sum v_i\otimes v_i^* & d_{V}: & V^*\otimes V\rightarrow \C, \text{
    given by }
  f\otimes w \mapsto f(w)\\
  b_{V}' :& \C \rightarrow V^*\otimes V, \text{ given by } 1 \mapsto
  \sum K^{r-1}v_i \otimes v_i^* & d_{V}': & V\otimes V^*\rightarrow
  \C, \text{ given by } v\otimes f \mapsto f(K^{1-r}v)
\end{align*}
are duality morphisms of $\cat$.  
In \cite{Oh} Ohtsuki defines an $R$-matrix operator defined on
$V\otimes W$ by
\nota{}
\begin{equation}\label{eq:R}
  R=\qr^{H\otimes H/2} \sum_{n=0}^{\ro-1} \frac{\{1\}^{2n}}{\{n\}!}\qr^{n(n-1)/2}
  E^n\otimes F^n.
\end{equation}
where $q^{H\otimes H/2}$ is the operator given by  
$$q^{H\otimes H/2}(v\otimes v') =q^{\lambda \lambda'/2}v\otimes v'$$
for weight vectors $v$ and $v'$ of weights of $\lambda$ and
$\lambda'$.  The braiding $c_{V,W}:V\otimes W \rightarrow W \otimes V$
on $\cat$ is defined by $v\otimes w \mapsto \tau(R(v\otimes w))$ where
$\tau$ is the permutation $x\otimes y\mapsto y\otimes x$.

For each $n \in \{0,\ldots,r-1\}$ let $S_n$ be the usual
$(n+1)$-dimensional irreducible highest weight $\Ubar$-module with
highest weight $n$.  The module $S_n$ has a basis $\{s_i=F^is_0 |
i=0,\ldots,n\}$ determined by $H.s_i=(n-2i)s_i$, $E.s_0=0=F^{n+1}.s_0$
and $E.s_i=\frac{\qn i\qn{n+1-i}}{\qn1^2}s_{i-1}$.  Its quantum
dimension is given by the trace of the action of $K^{1-r}$ and so
$\qdim(S_n)=(-1)^n\frac{\qn {n+1}}{\qn1}$.

Since $q$ is a root of unity and $F^r=0$ we can
consider a larger class of finite dimensional highest weight modules:
for each $\alpha\in \C$ we let $V_\alpha$ be the $r$-dimensional
highest weight $\Ubar$-module of highest weight $\alpha + r-1$.  The
modules $V_\alpha$ has a basis $\{v_0,\ldots,v_{r-1}\}$ whose action is
given by
\begin{equation}\label{E:BasisV}
H.v_i=(\alpha + r-1-2i) v_i,\quad E.v_i= \frac{\qn i\qn{i-\alpha}}{\qn1^2}
v_{i-1} ,\quad F.v_i=v_{i+1}.
\end{equation}
All the modules $V_\alpha$ have a vanishing quantum dimensions.  They
are divided into typical and atypical modules:
\begin{description}
\item[Atypical modules] If $k \in \Xr=\Z\setminus r\Z\subset \C$ then
  $V_k$ is indecomposable but not irreducible, however it is still
  absolutely irreducible (i.e. $\End_\cat(V_k)= \C\Id_{V_k}$ since any
  endomorphism must map the highest weight vector $v_0$ to a multiple
  of itself).  In particular, if $k\in \{0,\ldots,r-1\}$ then the
  assignment sending the highest weight vector $s_0$ of $S_{r-1-k}$ to
  the vector $v_k$ of $V_k$ determines an injective homomorphism
  $\imath: S_{r-1-k} \to V_k$.  Here the submodule $S_{r-1-k}$ in
  $V_k$ is not a direct summand.  Also, if $j\in \{1-r,\ldots,0\}$
  then the assignment sending the highest weight vector $v_0$ of $V_j$
  to the highest weight vector $s_0$ of $S_{r-1+j}$ induces a
  surjective homomorphism $\pi: V_j\to S_{r-1+j}$.
\item[Typical modules] If $\alpha\in \C\setminus \Xr$ then $V_\alpha$
  is irreducible and so absolutely irreducible.  We call such modules
  \emph{typical}.  
\end{description}

Let $\A$ be the set of typical modules.  For $g\in\C/2\Z$, define
$\cat_{g}$ as the full sub-category of weight modules with weights
congruent to $g$ mod $2$.  Then it is easy to see that
$\{\cat_g\}_{g\in \C/2\Z}$ is a $\C/2\Z$-grading in $\cat$ (see
\cite{CGP}).

\subsection{The link invariants $F$ and $F'$}\label{sub:FF'} 
The well-known Reshetikhin-Turaev construction defines a $\C$-linear
functor $F$ from the category of $\cat$-colored ribbon graphs with
coupons to $\cat$.  When $L$ is a $\cat$-colored framed link then
$F(L)$ can be identified with a complex number.  When $L$ is a framed
link whose components are all colored by $S_{n}$ then $F(L)$ is the
Kauffman bracket with variable specialization
$A=q^{1/2}=\e^{i\pi/2r}$, so it is a version of the colored Jones
polynomial specialized at the root of unity $q=\e^{i\pi/r}$ (for
details see Section \ref{s:Jones}).

Vanishing quantum dimensions make the functor $F$ trivial on any
closed $\cat$-colored ribbon graph that have at least one edge colored
by a typical module.  In \cite{GPT}, the definition of $F$ is extended
to a non-trivial map $F'$ defined on closed $\cat$-colored ribbon
graphs with at least one edge colored by a typical module.  We will
now recall how one can compute $F'$.

Let $T_W$ be any $\cat$-colored (1-1)-ribbon graphs with both ends
colored by the same element $W$ of $\cat$.  If $W$ is absolutely
irreducible then $F(T_W)$ is an endomorphism of $W$ that is determined
by a scalar $\brk{T_W}$:
$$F(T_W)=\brk{T_W}\Id_{W}.$$
Let $L$ be a closed $\cat$-colored ribbon graph with an edge colored
by a typical module $V_\alpha$.  By cutting this edge we obtain a
$\cat$-colored (1-1)-tangle $T_{V_\alpha}$ whose open edges are
colored by $V_\alpha$.  Then we define
$F'(L)=\qd(V_\alpha)\brk{T_{V_\alpha}}$.  It can be shown that $F'(L)$
does not depend on the choice of the edge to be cut and yields a well
defined invariant of $L$ (see \cite{GPT}).

We will use the following proposition latter.  
\begin{proposition} \label{P:SandV} Let $T$ be a (1-1)-tangle formed
  from a closed $\cat$-colored ribbon graph and a single open
  uncolored component.  Let $T_W$ be $T$ where the open component is
  colored by $W$.
  We have the following equality of scalars:
  $$\brk{T_{S_{r-1-k}}}=\brk{T_{V_k}}, \text{ for }k \in \{0,\ldots,r-1\}$$
  and 
  $$\brk{T_{S_{r-1+j}}}=\brk{T_{V_j}}, \text{ for }j \in \{1-r,\ldots,0\}.$$
\end{proposition}
\begin{proof}
  In this proof we use the language of coupons, for more details see
  \cite{Tu}.  In particular, a morphism $f:V\to W$ can be represented
  by a coupon $c(f)$, which is a box with arrows:
  $c(f)=\def\objectstyle{\scriptstyle} \def\labelstyle{\scriptstyle}
  \vcenter{\xymatrix @1 @-1.2pc 
    {\ar[d]^{W}\\ *+[F]\txt{\,$f$\, } \ar[d]^{V}\\ \: }} $.  By
  definition of $F$, we have $F(c(f))=f$.  By fusing this coupon to
  the bottom of the (1-1)-tangle $T_W$ we obtain a ribbon graph which
  we denote by $ T_W\circ c(f)$.  Similarly, we can fuse $c(f)$ to the
  top of the tangle $T_V$ to obtain a ribbon graph $c(f)\circ T_V$.

  From the discussion above about atypical modules we have the
  injection $\imath: S_{r-1-k} \to V_k$ and surjection $\pi: V_j\to
  S_{r-1+j}$, for $k \in \{0,\ldots,r-1\}$ and $j \in
  \{1-r,\ldots,0\}$.  Thus, as explained in the previous paragraph we
  can consider the ribbon graphs $T_{V_k}\circ c({\imath})$ and
  $c({\imath}) \circ T_{S_{r-1-k}}$.  Since the category of
  $\cat$-colored ribbon graphs is a ribbon category we have that
  $T_{V_k}\circ c({\imath})$ and $c({\imath}) \circ T_{S_{r-1-k}}$ are
  equal as ribbon graphs, so their images are equal under $F$.
  Combining this with the fact that $F(T_{V_k})$ and
  $F(T_{S_{r-1-k}})$ are scalars endomorphisms we have
  $$\brk{T_{V_k}} \imath=\brk{T_{V_k}} F(c({\imath}))=F(T_{V_k}\circ c({\imath}))
  =F(c({\imath}) \circ T_{S_{r-1-k}})=\imath \brk{ T_{S_{r-1-k}}}.$$ Thus,
  we have $\brk{T_{V_k}} =\brk{ T_{S_{r-1-k}}}.$ Similarly, we have
  $\brk{T_{S_{r-1+j}}}=\brk{T_{V_j}}$.
\end{proof}

\subsection{Comparison with the Jones polynomials}\label{s:Jones}
In this paper, by the colored Jones polynomial, we mean the Kauffman
bracket version which is an invariant of framed oriented links
independent of their orientation.  Let $L=L_1\sqcup \cdots \sqcup L_k
\subset S^3$ be a framed link and $J(L)\in \C[q^{\pm \frac{1}{2}}]$ be
its \emph{Jones polynomial} determined by the following skein
relations:
\begin{equation}
  \label{eq:kauffman1}
  q^{\frac{1}{2}}J\left(\epsh{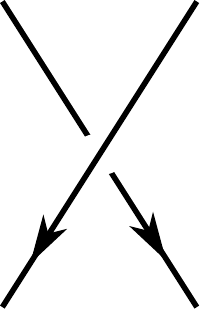}{9ex}\right)-q^{-\frac{1}{2}}
  J\left(\epsh{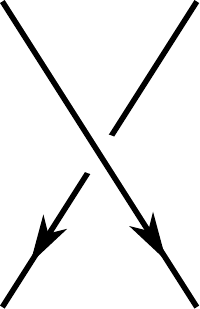}{9ex}\right)=
  (q-q^{-1})J\left(\epsh{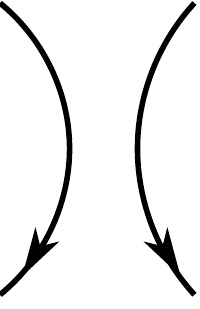}{9ex}\ \right), 
\end{equation}
\begin{equation}
  \label{eq:kauffman2}
  J\left(\epsh{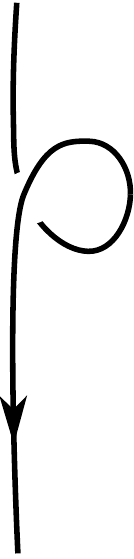}{9ex}\right)=
  -q^{3/2}J\left(\epsh{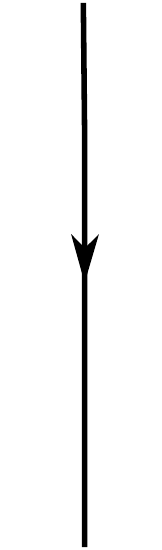}{9ex}\ \right) \ \  \text{and}\ \ 
  J\left({\epsh{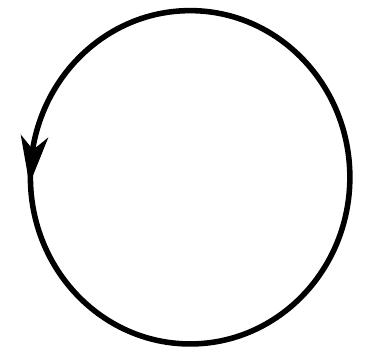}{6ex}}\right)= -q-q^{-1}. 
\end{equation}

More generally, if each $L_i$ is colored by an integer $n_i\geq 0$,
then roughly speaking, one defines the $\vec{n}^{th}$-colored Jones
polynomial $J_{\vec{n}}(L)$ as a linear combination of Jones
polynomials of links obtained by taking parallels of each component of
$L$ at most $n_i$ times.
More precisely, one identifies the tubular neighborhood of each component
$L_i$ with the product $S^1\times [-1,1]\times [-1,1]$ (using the
framing of $L_i$ and an arbitrary orientation), and defines links
$L_i^k=S^1\times \{\frac{0}{k},\frac{1}{k},\ldots
,\frac{k-1}{k}\}\times \{0\}$ and adopts the notation that $L_i^k\cdot
L_i^{h}=L_i^{k+h}$. Then one recursively defines a linear combination
of links parallel to $L_i$ as follows:
\begin{equation}\label{eq:tchebychev}
  T_n(L_i):=L^1_i\cdot T_{n-1}(L_i)-T_{n-2}(L_i)\ \ \ \  
  {\rm and}\ \  \ T_0(L_i)=\emptyset,\ T_1(L_i)=L_i.
\end{equation}
Finally, $J_{\vec{n}}(L)$ is defined as the linear combination of the
Jones polynomials of the links obtained by replacing $L_i$ with
$T_{n_i}(L_i)$. Clearly the above defined Jones polynomial corresponds
to the case when $n_i=1$ for all $ i$. The following holds:
\begin{prop}\label{P:JonesKauf}
  Let $L=L_1\sqcup \cdots \sqcup L_k\subset S^3$ be a framed oriented
  link and let $\vec{n}=(n_1,...,n_k)$ be a tuple of integers all
  greater or equal to 0.  Let $L_{\vec{n}}$ be the link $L$ such that
  $L_{i}$ is colored by $n_i$, for all $i$.  Similarly, let $L_S$ be
  the link $L$ such that $L_i$ is colored by $S_{n_i}$ where $S_{n_i}$
  is the simple module defined in Subsection \ref{SS:QuantMod}.  Then
  $$J_{\vec{n}}(L_{\vec{n}})|_{q=\exp(\frac{i\pi}{r})}=F(L_S).$$
\end{prop}
\begin{proof}
  First, assume that $n_i=1$, for all $i$.  In this case, we will
  prove that the relations of Equations \eqref{eq:kauffman1} and
  \eqref{eq:kauffman2} hold.
  We start by recalling that $S_1$ is spanned by two vectors $s_0,s_1$
  with $H(s_i)=1-2i, \ K(s_i)=q^{1-2i}s_i$ and $E(s_1)=0=F(s_2)$ while
  $E(s_2)=s_1$ and $F(s_1)=s_2$.  The second relation of Equation
  \eqref{eq:kauffman2} is a consequence of the formula for the quantum
  dimension $\qdim(S_1)=-q-q^{-1}$ given above.  The first relation
  follows from the fact that the inverse of the twist on $S_1$ is
  given by the action of $\theta=K^{\ro-1}\sum_{n=0}^{\ro-1}
  \frac{\{1\}^{2n}}{\{n\}!}\qr^{n(n-1)/2} (-KF)^n\qr^{-H^2/2}E^n$ (see
  \cite{CGP}).  To see that Equation \eqref{eq:kauffman1} holds,
  recall the braiding $c_{S_1,S_1}$ is defined by $v\otimes w \mapsto
  \tau(R(v\otimes w))$ where $R$ is the $R$-matrix and $\tau$ is the
  permutation $x\otimes y\mapsto y\otimes x$.  Since $E^2$ and $F^2$
  act by zero on $S_1$ we have $c_{S_1,S_1}$ and $c_{S_1,S_1}^{-1}$
  are determined by $\tau\circ \left(q^{\frac{H\otimes
        H}{2}}(\Id+(q-q^{-1})E\otimes F)\right)$ and
  $(\Id-(q-q^{-1})E\otimes F)q^{-\frac{H\otimes H}{2}}\circ \tau$,
  respectively.
  Thus, Equation \eqref{eq:kauffman1} follows from the following
  direct computations:
  $$(q^{\frac{1}{2}}\tau\circ R-q^{-\frac{1}{2}} R^{-1}\circ
  \tau)(s_0\otimes s_0)=(q-q^{-1}) s_0\otimes s_0,\
  (q^{\frac{1}{2}}\tau\circ R-q^{-\frac{1}{2}}R^{-1}\circ
  \tau)(s_0\otimes s_1)=(q-q^{-1}) s_0\otimes s_1,$$
  $$(q^{\frac{1}{2}}\tau\circ R-q^{-\frac{1}{2}}R^{-1}\circ \tau) 
  (s_1\otimes s_1) =(q-q^{-1}) s_1\otimes s_1,\
  (q^{\frac{1}{2}}\tau\circ R -q^{-\frac{1}{2}}R^{-1}\circ \tau)
  (s_1\otimes s_0)=(q-q^{-1})s_1\otimes s_0.$$ 
 
  Finally, to prove the statement in general it is sufficient to remark
  that the standard tensor decomposition of $S_1^{\otimes n}$ as a sum
  of copies of $S_i$ with $i\leq n$ is still valid for $n<r$ in
  $\cat$. To prove this it is sufficient to remark that if $2\leq n<r$
  then $S_{n-1}\otimes S_1\simeq S_{n}\oplus S_{n-2}$ and arguing by
  induction.  Hence the formula \eqref{eq:tchebychev} expressing
  $T_n(L)$ translates this decomposition algebraically expressing
  $F(L)$ (with $L$-colored by $n$) as a linear combination of
  invariants of cables of $L$ whose components are all colored by
  $S_1$.  Thus, the theorem follows.
\end{proof}

\subsection{The 3-manifold invariants $\Nr^0$ and $\WRT$}
In this subsection, we fix an integer $r\ge2$ with $r\notin4\Z$.  We start by
recalling some definitions given in \cite{CGP}.  Let $M$ be a compact
connected oriented 3--manifold, $T$ a $\cat$-colored ribbon graph in $M$ and
$\coh\in H^{1}(M\setminus T,\C/2\Z)$.  Let $L$ be an oriented framed link in
$S^3$ which represents a surgery presentation of $M$.  The map $g_\coh$
defined on the set of edges of $L\cup T$ with values in $\C/2\Z$ defined by
$g_\coh(e_i)=\coh(m_i)$, where $m_i$ is a meridian of $e_i$, is called the
\emph{$\C/2\Z$-coloring} of $L\cup T$ induced by $\coh$.

\begin{definition}\label{def:adm} 
  Let $M$, $T$ and $\coh$ be as above.
  \begin{enumerate}
  \item We say that $(M,T,\coh)$ is {\em a compatible triple} if for each edge
    $e$ of $T$ its coloring is in $\cat_{g_\coh(m_e)}$ where $m_e$ is a
    meridian of $e$.
  \item A compatible triple is \emph{$T$-admissible} if there exists
    an edge of $T$ colored by $V_\alpha\in \A$.
  \item A link $L\subset S^3$ which is a surgery presentation for a
    compatible triple $(M,T,\coh)$
    is \emph{computable} if one of the two following conditions holds:
  \begin{enumerate}
  \item $L\neq \emptyset$ and $g_\coh(L_i) \in ({\C/2\Z})\setminus
    (\Zd)$ for all components $L_i$ of $L$, or
  \item $L=\emptyset$ and there exists an edge of $T$ colored by
    $V_\alpha\in \A$.
  \end{enumerate}
\end{enumerate}
\end{definition}

Recall the set $\Hr=\{1-r,3-r,\ldots,r-1\}$ defined in \eqref{eq:Hr}.
For $\alpha\in \C\setminus\Z$ we define the Kirby color
$\Omega_{{\alpha}}$ as the formal linear combination
\begin{equation}
  \label{eq:Om}
  \Omega_{{\alpha}}=\sum_{k\in \Hr}\qd(\alpha+k)V_{\alpha+k}.
\end{equation}
If $\wb{\alpha}$ is the image of $\alpha$ in $\C/2\Z$ we say
that $\Omega_{{\alpha}}$ has degree $\wb{\alpha}$.  We can ``color'' a
knot $K$ with a Kirby color $\Omega_{{\alpha}}$: let
$K({\Omega_{{\alpha}}})$ be the formal linear combination of knots
$\sum_{k\in \Hr} \qd(\alpha+k) K_{\alpha+k}$ where $K_{\alpha+k}$ is the
knot $K$ colored with $V_{\alpha+k}$.  If
$\wb{\alpha}\in\C/2\Z\setminus\Zd$, by $\Omega_{\wb{\alpha}}$, we
mean any Kirby color of degree $\wb\alpha$.
Let $\Delta_-$ and $\Delta_+$ be the scalars given by:
  $$\Delta_-=\wb{\Delta_+}=\left\{
  \begin{array}{lll}
    i(r\qr)^{\frac32}&&
    \text{ if $\ro\equiv1$ mod }4\\
  (i-1)(r\qr)^{\frac32}&&
    \text{ if $\ro\equiv2$ mod }4\\ 
    - (r\qr)^{\frac32}&&
    \text{ if $\ro\equiv3$ mod }4.\\
  \end{array}\right. $$
Next we recall the main theorems of \cite{CGP}.  

\begin{theorem}[\cite{CGP}]\label{T:sl2CompSurgInv}
  If $L$ is a link which gives rise to a computable surgery
  presentation of a compatible triple $(M,T,\coh)$ then
  $$\Nr(M,T,\coh)=\dfrac{F'(L\cup T)}{\Delta_+^{p}\ \Delta_-^{s}}$$
  is a well defined topological invariant (i.e. depends only of the
  homeomorphism class of the triple $(M,T,\coh)$), where $(p,s)$ is
  the signature of the linking matrix of the surgery link $L$ and for
  each $i$ the component $L_i$ is colored by a Kirby color of degree
  $g_{\coh}(L_i)$.  
\end{theorem}

For $\alpha \in \srcol$, let $u_\alpha$ be the unknot in $S^3$ colored
by $V_\alpha$.  Let $\coh_\alpha$ be the unique element of
$H^1(S^3\setminus u_\alpha;\C/2\Z)$ such that
$(S^3,u_\alpha,\coh_\alpha)$ is a compatible triple.  
\begin{theorem}[\cite{CGP}]\label{T:N0sl2} 
  Let $(M,T,\coh)$ be a compatible triple.  Define
  $$
  \Nr^{0}(M,T,\coh)=\frac{\Nr((M,T,\coh)\#(S^3,u_\alpha,\coh_\alpha))}
  {\qd(\alpha)}.
  $$ 
  Then $\Nr^{0}(M,T,\coh)$ is a well defined topological invariant
  (i.e. depends only of the homeomorphism class of the compatible
  triple $(M,T,\coh)$).  Moreover, if $(M,T,\coh)$  has a computable surgery
  presentation then $\Nr^0(M,T,\coh)=0$.
\end{theorem}

Let us also give a definition of the refined Witten-Reshetikhin-Turaev
invariants $\WRT(M,T,\coh)$.  The definition is based on the fact that
the Kauffman bracket version of the colored Jones polynomial can be
computed through $F$ 
 (see Proposition \ref{P:JonesKauf}).  

We define the Kirby colors of degree
$\wb0$ and $\wb1$ respectively by
$$\Omega^{RT}_0:=\sum_{0\le j\le r-2}^{j\ \rm{even}} \qN{j+1}S_j 
\quad\text{and}\quad\Omega^{RT}_1:=\sum_{0\le j\le r-2}^{j\ \rm{odd}}
-\qN{j+1}S_j$$
\begin{lemma} \label{L:Dsot}Let $\Delta^{\sot}_\pm=F(u_{\pm1})$ where $u_{\pm1}$ is
  the unknot with framing $\pm1$ colored by $\Omega^{RT}_0$.  Then 
  \[\Delta^{\sot}_+=\frac{\Delta_+}{\qn1r}\quad\text{and} 
  \quad\Delta^{\sot}_-=\wb{\Delta^{\sot}_+}=-\frac{\Delta_-}{\qn1r}.\]
  In particular, in both case, $\Delta^{\sot}_\pm\neq0$.
\end{lemma}
\begin{proof}
  The proof is a direct computation using the values of the quantum
  dimension and of the twist for the simple modules $S_n$.  In
  particular, we have $\qdim(S_i)=(-1)^i\dfrac{\qn{i+1}}{\qn1}$ and
  the twist on $S_i$ acts by the scalar $(-1)^iq^{\frac{i^2+2i}2}$.  Thus,
  $$\Delta^\sot_+=\qn1^{-2}\sum_{j=0,\ j\ \rm{even}}^{r-2}\qn{j+1}^2q^{\frac{j^2+2j}2}
  =\qn1^{-2}(q^2\Sigma_3+q^{-2}\Sigma_{-1}-2\Sigma_1)$$ where
  $\ds{\Sigma_a=\sum_{n=0}^{\lfloor\frac{r-2}2\rfloor}q^{2(n^2+an)}}$ is
  part of a quadratic Gauss sum.  These terms can be computed using
  standard results on quadratic Gauss sum.
\end{proof}

Kirby and Melvin \cite{KM} and Blanchet \cite{Bl} consider invariants of $(M,\emptyset,\coh)$ where $\coh \in H^1(M,\Zd)$.  The following theorem is a slight generalization of these invariants (here we use the conventions of this paper and not the conventions of \cite{KM,Bl}).    
\begin{theorem}[Refined Witten-Reshetikhin-Turaev invariants]\label{T:RefinedWRT}
Let $(M,T,\omega)$ be a compatible triple with $T$ a
$\cat_{\wb0}\cup\cat_{\wb1}$-colored ribbon graph and $\omega\in
H^1(M\setminus T,\Zd)$.  If $L$ is a link which gives rise to a surgery presentation of the
  pair $(M,T)$ then
  $$\WRT_r(M,T,\coh)=\frac{F(L\cup T)}
  {(\Delta^{\sot}_+)^{p}(\Delta^{\sot}_-)^{s}}$$
 is a well defined topological invariant (i.e. depends only of the
  homeomorphism class of the triple $(M,T,\coh)$), where $(p,s)$ is
  the signature of the linking matrix of the surgery link $L$ and for
  each $i$ the component $L_i$ is colored by a Kirby color of degree
  $g_{\coh}(L_i)$.
\end{theorem}
\begin{proof}
  In \cite{KM}, for $T=\emptyset$ and $r$ even, this invariant is
  considered in a slightly different form.  Also, in Remark II.4.3 of
  \cite{Bl} for $T=\emptyset$, the existence of this invariant is
  discussed.  Indeed, the Reshetikhin-Turaev functor applied on graphs
  colored by the module $S_1\in\cat_{\wb 1}$ satisfies the Kauffman
  skein relation for $A=q^{\frac12}=\exp(\frac{i\pi}{2r})$.  It
  follows that if $L\subset S^3$ is a framed link whose components are
  colored by elements of $\{S_0,\ldots,S_{r-2}\}$ then $F(L)$ is the
  meta-bracket (see \cite{Bl,BHMV}) evaluated at the element
  corresponding to the coloring of $L$  
  at $A=q^{\frac12}$.  It
  follows that $\WRT_r(M,\emptyset,\coh)$ is the invariant denoted
  $\theta_{q^{\frac12}}(M_{L,g_\omega})$ in Remark II.4.3 of
  \cite{Bl}.

  For a complete proof of the theorem, one can also apply Theorem 3.7 of
  \cite{CGP} to the modular category obtained as the
  quotient of the subcategory of $\cat$ generated by $S_1$ 
  by its
  ideal of projective modules. 
   Indeed, this category is obviously a
  $\Zd$-modular category relative to $\emptyset$ with modified
  dimension $\qdim$ and trivial periodicity group.
\end{proof}

In particular, when $\coh=0$ one gets an invariant of manifolds also
known as the $SO(3)$ version of the Reshetikhin-Turaev invariants:
\begin{Df}
  Let $T$ be a $\cat_{\wb 0}$-colored ribbon graph in a closed
  3-manifold $M$ then
  $$WRT^\sot(M,T)=WRT(M,T,0).$$
\end{Df}

\begin{remark}
  Let us call $WRT^\sud(M,T)$ the original WRT-invariant which is
  obtained as in Theorem \ref{T:RefinedWRT} except that all components
  of $L$ are colored by $\Omega^{RT}=\Omega^{RT}_0+\Omega^{RT}_1$ (and
  the elements $\Delta^\sud_\pm$ are also defined with $\Omega^{RT}$).
  For odd $r$, it can be shown that $WRT(M,T,\coh)$ depends weakly of
  the compatible cohomology class $\coh\in H^1(M\setminus T,\Zd)$.
  Similarly, $WRT^\sud(M,T)$ is proportional to $WRT^\sot(M,T)$ (see
  \cite[Section III]{Bl}). A similar property holds for $\Nr^0$ and
  more generally, for admissible 
   triples: 
   $\Nr(M,T,\coh)$ depends
  essentially only of the reduction modulo $\Z$ of the compatible
  cohomology class $\coh\in H^1(M\setminus T,\C/2\Z)$.
  \\
  The behavior for $r$ even is different: in this case results of
  \cite{KM,Bl} suggest the following conjecture:
  $$WRT^\sud(M,T)=\sum_{\text{compatible }\coh\in H^1(M\setminus T,\Zd)}WRT(M,T,\coh).$$
\end{remark}

\section{Relations between $F'$ and the colored Jones polynomial}\label{S:RelationsFandJones}
Recall the $\ro$-dimensional modules $V_\alpha$, $\alpha \in \C$,
given in Subsection \ref{SS:QuantMod}.  Using the basis given in
Equation \eqref{E:BasisV} and its dual basis we identify $V_\alpha$
and $V_\alpha^*$ with $\C^\ro$.  With these identifications we can
identify certain $\Hom$-spaces with spaces of matrices.  For example,
we can make the following identifications: $\End_\cat(
V_\alpha)=Mat_{\ro \times \ro}(\C)$ and $\Hom(\C, V_\alpha \otimes
V_\alpha^*)=Mat_{1 \times \ro^2}(\C)$.

We say a function $g:\C\to \C$ is a \emph{Laurent polynomial in $q^\alpha$} if there exists a Laurent polynomial $f\in \C[x,x^{-1}]$ such that $g(\alpha)=f(q^\alpha)$.   The action of the basis given in Equation \eqref{E:BasisV} implies that all the  entries in the matrices $\rho_{V_\alpha}(E), \rho_{V_\alpha}(F), \rho_{V_\alpha}(H)$ and $\rho_{V_\alpha}(K)$ are Laurent polynomial in $q^\alpha$.  
\begin{lemma}\label{L:LaurPoly}
All the entries in the image of the maps 
\begin{align*}
g_b & :\C\to Mat_{1 \times \ro^2}(\C) \text{ given by }\alpha \mapsto b_{V_\alpha},\\
g_d & : \C\to Mat_{\ro^2\times 1}(\C) \text{ given by }\alpha \mapsto d_{V_\alpha},\\
g_{b'} &: \C\to Mat_{1 \times \ro^2}(\C) \text{ given by }\alpha \mapsto b'_{V_\alpha},\\
g_{d'} & :\C\to Mat_{\ro^2\times 1}(\C) \text{ given by }\alpha \mapsto d'_{V_\alpha}.
\end{align*}
are Laurent polynomials in $q^\alpha$.  Also, for each entry $f_{ij}$
in the image of the map $f:\C\times \C \to Mat_{\ro^2\times
  \ro^2}(\C)$, $(\alpha,\beta) \mapsto q^{-\alpha
  \beta/2}q^{-(r-1)(\alpha+\beta)/2}c_{V_\alpha,V_\beta}$ there exists
a two variable Laurent polynomial $g_{ij}(x,y)$ such that
$f_{ij}(\alpha,\beta)=g_{ij}(q^\alpha,q^\beta)$.
\end{lemma}
\begin{proof}
The first statement follows from the formulas for $b, d, b'$ and  $d'$ given in Subsection \ref{SS:QuantMod}.  For example,  the entry in image of $g_{d'}$ corresponding to $v_i\otimes v_j^*$ is $v_j^*(K^{1-r}v_i)= \delta_{ij}q^{(1-r)(\alpha + r-1-2i)}$.  The second statement follows from the form of the $R$-matrix given in Equation \eqref{eq:R}.  
In particular, if $v_i$ and $v_j$ are any basis vectors of $V_\alpha$ and $V_\beta$, respectively then  
$$q^{-\alpha \beta/2}q^{-(r-1)(\alpha+\beta)/2}\qr^{H\otimes H/2}E^n\otimes F^n.v_i\otimes v_j 
=q^{-\alpha(j+n)-\beta(i-n)}q^{c/2}E^n\otimes F^n.v_i\otimes v_j.$$
where $c$ is an integer which does not depend on $\alpha$ or $\beta$.
Also, $$E^n\otimes F^n(v_i\otimes v_j)=
\frac{\{i\}!}{\qn{i-n}!\qn1^{2n}}
{\qn{i-\alpha}\qn{i-1-\alpha}\cdots\qn{i-(n-1)-\alpha}}v_{i-n}\otimes
v_{j+n}.$$ Since the coefficients in the last two equalities are
Laurent polynomial in $q^\alpha$ and $q^\beta$, the desired result
about the function $f$ follows.
\end{proof}

 The above lemma has the following corollaries.

\begin{corollary}\label{C:TangleHolom}
  Let $T_{(V_{\alpha_1},\dots, V_{\alpha_n})}$ be a (1-1)-tangle with
  $n$ components whose $i^{th}$ component is colored by
  $V_{\alpha_i}$, $\alpha_i\in \C$.  Then the function $g_T:\C^n\to
  \C$ given by $(\alpha_1,\dots, \alpha_n)\mapsto
  \brk{T_{(V_{\alpha_1},\ldots, V_{\alpha_n})}}$ is a holomorphic
  function in $\C^n$.  In particular $g_T$ is continuous.
\end{corollary}
\begin{proof}
Assume the $1^{st}$ component is the open component.  By definition we have 
$$F(T_{(V_{\alpha_1},\dots, V_{\alpha_n})})= \brk{T_{(V_{\alpha_1},\dots, V_{\alpha_n})}}\Id_{V_{\alpha_1}}$$
so it is enough to consider $F(T_{(V_{\alpha_1},\dots,
  V_{\alpha_n})})$.  The value of $F(T_{(V_{\alpha_1},\dots,
  V_{\alpha_n})})$ is computed by decomposing a projection of
$T_{(V_{\alpha_1},\dots, V_{\alpha_n})}$ into building blocks made of
cups, caps, vertical edges and crossings.  Then the building blocks
are associated with the duality morphisms, identity and the positive
and negative braidings, respectively.  These morphisms are tensored
and composed according to the projection of $T_{(V_{\alpha_1},\dots,
  V_{\alpha_n})}$.  Lemma \ref{L:LaurPoly} implies that the
contributions from a duality morphism corresponding to a cup or cap on
the $i^{th}$ component is a Laurent polynomial in $q^{\alpha_i}$.
Lemma \ref{L:LaurPoly} also implies the all contributions of a
crossing between the $i^{th}$ and $j^{th}$ components are Laurent
polynomials in $q^{\alpha_i}$ and $q^{\alpha_j}$ times a factors of
$q^{-\alpha_i \alpha_j/2}q^{-(r-1)(\alpha_i+\alpha_j)/2}$.  Thus, the
map $g_T(\alpha_1,\dots, \alpha_n)= \brk{T_{(V_{\alpha_1},\dots,
    V_{\alpha_n})}}$ is a Laurent polynomials in the variables
$q^{\alpha_1},\dots, q^{\alpha_n}$ times a integral powers of
$q^{-\alpha_i \alpha_j/2}q^{-(r-1)(\alpha_i+\alpha_j)/2}$ and so $g_T$
is holomorphic.
\end{proof}

\begin{corollary}\label{C:LaurentPoly}
  Let $K$ be a knot.  Let $K^f_{V_\alpha}$ be $K$ colored by
  $V_\alpha$ with framing $f\in \Z$.  Let $T_{V_\alpha}^0$ be a
  (1-1)-tangle with zero framing whose closure is $K_{V_\alpha}^0$.
  Then there exists a Laurent polynomial $\wt K(X)\in\C[X,X^{-1}]$
  such that $\brk{T_{V_\alpha}^0}=\wt K(q^\alpha)$ and
  \begin{equation}
    \label{E:NtildeK}
    F'(K_{V_\alpha}^f)=\theta_\alpha^f\qd(\alpha)\wt K(q^\alpha)
  \end{equation} 
  where $\theta_\alpha=q^{\frac12(\alpha^2-(\ro-1)^2)}$ is the twist
  on $V_\alpha$.  Moreover,\\ $\wt K(q^{\alpha+r})=\wt K(q^\alpha)$,
  $F'(K_{V_{\alpha+2r}}^f)=q^{2r\alpha f}F'(K_{V_\alpha}^f)$ and
  $F'(K_{V_{\alpha+r}}^f)=(-1)^{r+1}(iq^\alpha)^{rf}F'(K_{V_{\alpha}}^f)$.
\end{corollary}
\begin{proof}
  As in the proof of Corollary \ref{C:TangleHolom} the function
  $g_T(\alpha)= \brk{T_{V_\alpha}^0}$ is a Laurent polynomial in
  $q^\alpha$ times an integral power of $q^{\alpha^2/2}$.  From the
  form of the map $g_c$ in Lemma \ref{L:LaurPoly} the integral power
  of $q^{\alpha^2/2}$ is equal to the number of positive crossing
  minus the number of negative crossing in the projection of
  $T_{V_\alpha}^0$.  Since the framing of $K_\alpha$ is zero this
  power is zero.  Thus, $g_T(\alpha)$ is a Laurent polynomials in
  $q^\alpha$ and so there exists a $\wt K(X)\in\C[X,X^{-1}]$ such that
  $\brk{T_\alpha^0}=\wt K(q^\alpha)$.  Now we can use the duality and
  the braiding to compute the value of the twist:
  $$\theta_\alpha=\brk{\epsh{fig16}{30pt}\put(-3,-7){\ms{\alpha}}}=q^{\frac12(\alpha^2-(\ro-1)^2)}.$$
  Then Equation \eqref{E:NtildeK} follows from the above discussion
  and the definition of $F'$:
  $$F'(K_{V_\alpha}^f)=\theta_\alpha^fF'(K^0_{V_\alpha})=\theta_\alpha^f\qd(\alpha)\brk{T_{V_\alpha}^0} =\theta_\alpha^f\qd(\alpha)\wt K(q^\alpha).$$  

  Next we will show that $\wt K(q^{\alpha+r})=\wt K(q^\alpha)$.
  Consider the one dimensional space $\sig=\C$ with the $\Ubar$-module
  structure given by
  $$Ev=Fv=0,  \;\; Hv=rv$$
  for any $v\in \sig$.  The quantum dimension of $\sig$ is
  $(-1)^{r+1}$.  From the form of the $R$-matrix we have:
\begin{equation}
  \label{eq:sigmabraid}
  \brk{\epsh{fig16}{30pt}\put(-3,-7){\ms{\sig}}}=-i^{-r} ,\quad 
  F\left( \put(5,17){$\sig$}\put(18,17){$\sig$}\epsh{fig26}{9ex}\right)=
  i^rF\left(\put(5,17){$\sig$}\put(19,17){$\sig$}\!\epsh{fig10}{9ex}\ \ \epsh{fig10}{9ex}
  \right)\quad\text{and}\quad
  F\left( \put(5,17){$\sig$}\put(12,17){$V_\alpha$}\epsh{fig26}{9ex}\right)=
  q^{(\alpha +r -1)r} F\left(\put(5,17){$\sig$}\put(12,17){$V_\alpha$}\epsh{fig27}{9ex}
  \right).
\end{equation}
Hence for a 0-framed knot $K$ colored with $\sig$, one has
$F(K)=F(\text{unknot})=(-1)^{r+1}$.\\
Let $T^0$ be the zero framed tangle underlying $T_{V_\alpha}^0$.  Let
$T^0_\sig$ be $T^0$ colored with $\sig$.  Since $T^0_\sig$ has zero
framing then $\brk{ T^0_{\sig}}=1$.  Now $F(T_{V_{\alpha +r}}^0)$ is
equal to the endomorphism associated to $T^0$ labeled with $V_\alpha
\otimes \sig$ or equivalently the 2-cabling of $T^0$ where the two
components are labeled by $V_\alpha$ and $\sig$, respectively.  We can
use the third equality in Equation \eqref{eq:sigmabraid} to unlink the
component labeled with $\sig$ from the component labeled with
$V_\alpha$.  Therefore, since $T_{V_{\alpha +r}}^0$ has zero framing
we have
$$\brk{ T_{V_{\alpha +r}}^0}=\brk{ T_{V_{\alpha}}^0}\brk{ T^0_{\sig}}=\brk{ T_{V_{\alpha}}^0}.$$

Finally, Equation \eqref{E:NtildeK} and the above formulas for $\theta_\alpha$ and $\qd(\alpha)$ imply:
 $$F'(K_{V_{\alpha+2r}}^f)=\theta_{\alpha+2r}^f\qd(\alpha+2r)\wt K(q^{\alpha+2r}) 
 =(q^{(2r\alpha +2r^2)}\theta_{\alpha})^f \qd(\alpha)\wt K(q^{\alpha})
 =q^{2r\alpha f}F'(K_{V_{\alpha}}^f)$$ and
 similarly $$F'(K_{V_{\alpha+r}}^f)=\theta_{\alpha+r}^f\qd(\alpha+r)\wt
 K(q^{\alpha+r}) =(-1)^{r+1}(iq^\alpha)^{r f}F'(K_{V_{\alpha}}^f)$$
 which concludes the proof.
\end{proof}
\begin{remark}\label{rem:symmetry}
  Corollary \ref{C:LaurentPoly} with Proposition \ref{P:SandV} imply
  the well known symmetry principle relating the colored Jones
  polynomial associated to $S_{k-1}$ with the one associated to
  $S_{r-1-k}$ for $k\in\{1,\ldots,r-2\}$.
\end{remark}

\begin{corollary}\label{cor:residuejones}
  Let $K$ be a knot and let $K_{V_\alpha}$ be $K$ colored by
  $V_\alpha$.  The function $g_K:\C\setminus \Xr \to\C$ defined by
  $\alpha \mapsto F'(K_{V_\alpha})$ is a meromorphic function on the
  whole plane $\C$.  Moreover, the residue at each pole is determined
  by the colored Jones polynomial.
\end{corollary}
\begin{proof}
  Recall that $$F'(K_{V_\alpha})= \qd(\alpha) \brk{T_{V_\alpha}}
  =(-1)^{r-1}\prod_{j=1}^{r-1} \frac{\{j\}}{\{\alpha+r-j\}}
  \brk{T_{V_\alpha}}$$ where $_{V_\alpha}$ is the (1-1)-tangle
  obtained from cutting $K_{V_\alpha}$.  From Corollary
  \ref{C:TangleHolom} it follows that $\alpha \mapsto
  (-1)^{r-1}\prod_{j=1}^{r-1} \{j\} \brk{T_{V_\alpha}}$ is a
  holomorphic function in the entire plane $\C$.  Also, it is clear
  that $\alpha \mapsto \prod_{j=1}^{r-1} \{\alpha+r-j\}$ is a
  holomorphic function in the entire plane $\C$ which is zero when
  $\alpha \in \Z\setminus r\Z$.  Therefore, the quotient of these two
  functions is a meromorphic function whose set of poles is $
  \Z\setminus r\Z$.

  All of these poles are simple and so the residue can be computed as
  follows.  Let $n\in \Z\setminus r\Z$.  The residue at $n$ of the
  $2r$-periodic meromorphic function $\qd$ is given by
  \newcommand{\Res}{\operatorname{Res}}
  $$\Res(\qd,n)=\lim_{\alpha\to n}
  (\alpha-n)(-1)^{r-1}\frac{r\qn\alpha}{\qn{r\alpha}}=\lim_{x\to 0}
  (-1)^{r-1}\frac{xr\sin\bp{\frac{\pi(n+x)}r}}{\sin\bp{{\pi(n+x)}}}
  =(-1)^{r-1+n}\frac r\pi\sin\bp{\frac{n\pi}r}.$$ 
  So the residue of
  $g_K$ at $n$ is equal to
  $$\Res(g_K,n)=\Res(\qd,n)\brk{T_{V_n}}
  =(-1)^{r-1+n}\frac r\pi\sin\bp{\frac{n\pi}r}\brk{T_{V_n}}.$$

 To finish the proof we will show that the above formula of $\Res(g_K,n)$ can be rewritten in terms of the colored Jones polynomial.
   To do this 
     we have two cases.  
     First, suppose $n=k+2mr$ with $k\in  \{1,\ldots,r-1\}$ and $m\in\Z$.  By Corollary \ref{C:LaurentPoly} and Proposition \ref{P:SandV} we have  
     $$\brk{T_{V_n}}=\brk{T_{V_k}}=\brk{T_{S_{r-1-k}}}.$$
     Combining the fact that  $\qdim(S_{r-1-k})=(-1)^{r-1-k}\frac{\qn{r-k}}{\qn{1}}=(-1)^{r-1-k}\frac{\sin\bp{\frac{k\pi}r}}{\sin\bp{\frac{\pi}r}}$ and Proposition~\ref{P:JonesKauf} we have  
       $$J_{r-1-k}(K)|_{q=\e^{i\pi/r}}=(-1)^{r-1-k}\frac{\sin\bp{\frac{k\pi}r}}{\sin\bp{\frac{\pi}r}}\brk{T_{S_{r-1-k}}}.$$
       Thus,
     $$\Res(g_K,n)=\frac r\pi\sin\bp{\frac{\pi}r}J_{r-1-k}(K)|_{q=\e^{i\pi/r}}.$$

  Similarly, if $n=k+2mr$ with $k\in \{1-r,\ldots,-1\}$ and $m\in\Z$
  then 
  one can show that
  $$\Res(g_K,n)=-\frac r\pi\sin\bp{\frac{\pi}r}J_{r-1+k}(K)|_{q=\e^{i\pi/r}}.$$
\end{proof}

\section{Surgery on a knot in the 3-sphere $S^3$}\label{S:ConjKnot}
In this section we prove Conjecture \ref{C:MainConj} when $M$ is an
empty closed manifold obtained by surgery on a non-zero framed knot in
$S^3$:

\begin{theorem}\label{ThMknot}
  Suppose that $K$ is a knot in $S^3$ with non-zero framing $f$.  Let
  $M$ be the manifold obtained by surgery on the knot $K$ and
  $\omega\in H^1(M,\Zd)$.  Then
  $$\Nr^0(M,\emptyset,\omega)=|f|\WRT(M,\emptyset,\omega)
  =\ord(H_1(M; \Z)) \WRT(M,\emptyset,\omega).$$
\end{theorem}
\begin{corollary}
  Let $M$ be a rational homology sphere obtained by surgery on a knot
  in $S^3$ then
  $$ \WRT^{\sot}(M,\emptyset)=\frac1{\ord(H_1(M; \Z))}\Nr^0(M,\emptyset,0).$$
\end{corollary}
\begin{remark}
  The three invariants $WRT$, $\Nr^0$ and $M\mapsto\ord(H_1(M,\Z))$
  are multiplicative 
  with respect
  to the connected sum of 3-manifolds.
  Hence Theorem \ref{ThMknot} implies that Conjecture \ref{C:MainConj}
  is also true for a connected sum of manifolds, each obtained by
  surgery on a knot in $S^3$.
\end{remark}
The rest of this section is devoted to the proof of Theorem \ref{ThMknot}.
\begin{proof}[Proof of Theorem \ref{ThMknot}.]
  First we improve the results of \cite[Section 2.4]{CGP} 
  and derive 
  a formula for $\Nr^0(M,\emptyset,\omega)$.  We still denote by
  $\omega$ the integer in $\{0,1\}$ whose class modulo $2$ is the
  value $g_\omega(K)$ of the cohomology class on the meridian of $K$
  and let $e\in\{0,1\}$ be such that $\wb e=\wb{r-1+\omega}\in\Zd$.

  For $\alpha \in \srcol$, recall the function $\PP(\alpha)=\sum_{k\in\Hr}F'(K_{V_{\alpha+k}})$
  of \cite[Section 2.4]{CGP} (as above, $K_V$ means $K$ colored by $V$).
  The function $\PP$  is 
  continuous and so can be naturally extended to all of $\C$.  
  Indeed, let $DK_{(V_\alpha,V_\beta)}$ be the 2-cable of $K$ whose
  components are colored with $V_\alpha$ and $V_\beta$ such that
  $\alpha$ or $\beta$ is in $\srcol$.
  From Lemma \ref{L:LaurPoly} we have that the map
  $(\alpha,\beta)\mapsto q^{-\frac
    f2(\alpha^2+\beta^2+2\alpha\beta)}F'(DK_{(V_\alpha,V_\beta)})$ is
  a rational function in
  $\frac1{q^{r\alpha}-q^{-r\alpha}}\C[q^{\pm\alpha},q^{\pm\beta}]\cap\frac1{q^{r\beta}-q^{-r\beta}}\C[q^{\pm\alpha},q^{\pm\beta}]$
  (also see the proof of Corollary \ref{C:TangleHolom}).  Thus, this
  function is a Laurent polynomial in
  $\C[q^{\pm\alpha},q^{\pm\beta}]$.  In addition, if $\alpha+\beta\in
  \srcol$ then $F'(DK_{(V_\alpha,V_\beta)})$ can be computed by
  coloring $K$ with $V_\alpha\otimes
  V_\beta\simeq\bigoplus_{k\in\Hr}V_{\alpha+\beta+k}$.  Combining the
  statements of this paragraph we have
  $$\PP(\alpha+\beta)=\sum_{k\in \Hr}
  F'\big(K_{V_{\alpha+\beta+k}}\big)=F'(DK_{(V_\alpha,V_\beta)})$$ is
  a continuous function of $(\alpha,\beta)$, which we extend to all of
  $\C\times \C$.

  Next we give a formula for $\Nr^0$ in terms of $\PP$.  By sliding
  the unknot $o_\alpha$ on $K$ we obtain a computable presentation of
  $(M,\emptyset,\omega)\#(S^3,o_\alpha,\omega_\alpha)$ as in Theorem
  \ref{T:N0sl2}.  This produces the link
  $DK_{(\Omega_{e-\alpha},V_\alpha)}$ where
  $\Omega_{e-\alpha}=\sum_{h\in \Hr} \qd(e-\alpha+h)V_{e-\alpha+h}$ is
  a Kirby color of degree $\wb{\omega-\alpha}$.  By definition of
  $\Nr^0$,
  $$\Nr^0(M,\emptyset,\omega)= \frac{1}{\Delta_{\sign(f)}\qd(\alpha)}\sum_{h\in \Hr}
  \qd(e-\alpha+h)F' \big(DK_{(V_\alpha,V_{e-\alpha+h})}\big).$$ Since
  $\qn{r(e-\alpha+h)}=\qn{-r(e-\alpha+h)}=(-1)^\omega\qn{r\alpha}$ we
  have
  $$\Delta_{\sign(f)}\Nr^0(M,\emptyset,\omega)= \frac{(-1)^\omega}{\qn\alpha}
  \sum_{h\in \Hr} \qn{\alpha-h-e}\PP(h+e)$$
  $$=\frac{(-1)^\omega q^{\alpha}}{q^\alpha-q^{-\alpha}}\sum_{h\in\Hr}q^{-h-e}\PP(h+e)
  -\frac{(-1)^\omega
    q^{-\alpha}}{q^\alpha-q^{-\alpha}}\sum_{h\in\Hr}q^{h+e}\PP(h+e).$$
  Finally, as  $\Nr^0(M,\emptyset,\omega)$
  does not depend on $\alpha$ we have
  $$\Nr^0(M,\emptyset,\omega)
  =\frac{(-1)^\omega}{\Delta_{\sign(f)}}\sum_{k\in\Hr}q^{k+e}\PP(k+e)
  =\frac{(-1)^\omega}{\Delta_{\sign(f)}}\sum_{k\in\Hr}q^{-k-e}\PP(k+e).$$

  Next we use the last formula and the continuity of $\PP$ to write a
  multiple of $\Nr^0$.  In particular, let $S$ be the following limit:
  \begin{align*}
    S=(-1)^\omega\Delta_{\sign(f)}\Nr^0(M,\emptyset,\omega)
  &=\lim_{\ve\to0}\sum_{\ell\in\Hr}q^{\ell+e}\PP(\ve+\ell+e)\\
  & = \lim_{\ve\to0}\sum_{k,\ell\in\Hr}q^{\ell+e}F'(K_{(\ve+k+\ell+e)})\\
    &=\lim_{\ve\to0}\sum_{n=1-\ro }^{\ro -1}\sum_{\tiny{\begin{array}{c}
          k,\ell\in\Hr\\k+\ell=2n \end{array}}}
    q^{\ell+e}F'(K_{(\ve+2n+e)})
  \end{align*}
  In this sum, for fixed $n$ the only part of the interior sum which
  varies is $q^\ell$ for $k,\ell\in\Hr$ with $k+\ell=2n$.  Here the
  possible values of $\ell$ are integers from $\max(1-\ro ,1-\ro +2n)$
  to $\min(\ro -1,\ro -1+2n)$ and so the sum of $q^\ell$, over these
  values, is equal to $q^{n}\dfrac{\qn{\ro
      -|n|}}{\qn1}=q^{n}\dfrac{\qn{|n|}}{\qn1}$.  Therefore, we have
  the following expression for $S$:
  \begin{align*}
    S&=
    \lim_{\ve\to0}\frac{1}{\qn1}\sum_{n=1-\ro }^{\ro -1}q^{n+e}{\qn{|n|}}F'(K_{V_{\ve+2n+e}})\\
    &=\lim_{\ve\to0}\frac{1}{\qn1}
    \sum_{n=1}^{\ro -1}\bp{{\qn{|n|}}q^{n+e}F'(K_{V_{\ve+2n+e}})
      +{\qn{|n-r|}}q^{n+e-r}F'(K_{V_{\ve+2n+e-2r}})}.
      \end{align*}
Now  Corollary \ref{C:LaurentPoly} and a direct computation show that
        \begin{align*}
         S   &=\lim_{\ve\to0}\frac{1}{\qn1}
    \sum_{n=1}^{\ro -1}F'(K_{V_{\ve+2n+e}}){\qn{n}}q^{n+e}\bp{1-q^{-2rf(\ve+2n+e)}}\\
    &=\frac{1}{\qn1}
    \sum_{n=1}^{\ro -1}\brk{T_{V_{2n+e}}}{\qn{n}}q^{n+e}\lim_{\ve\to0}\qd(\ve+2n+e)\bp{1-q^{-2rf\ve}}\\
    &=\frac{(-1)^{r-1}r}{\qn1}
    \sum_{n=1}^{\ro -1}\brk{T_{V_{2n+e}}}{\qn{n}}q^{n+e}\qn{2n+e}\lim_{\ve\to0}\frac{\qn{rf\ve}}{\qn{r\ve+re}}\\
    &=\frac{(-1)^\omega rf}{\qn1}
    \sum_{n=1}^{\ro -1}q^{n+e}{\qn{n}}\qn{2n+e}\brk{T_{V_{2n+e}}}.
  \end{align*}
  Coming back to $\Nr^0$, we have
  \begin{equation*}
    \Nr^0(M,\emptyset,\omega)=\frac{rf}{\qn1\Delta_{\sign(f)}}
    \sum_{n=1}^{\ro -1}q^{n+e}{\qn{n}}\qn{2n+e}\brk{T_{V_{2n+e}}}
    =c\sum_{n=0}^{r-1}\vp_e(2n+e).
  \end{equation*}
  where $c=\frac{rf}{\qn1\Delta_{\sign(f)}}$ and
  $\vp_e(k)=(q^{k}-q^e)\qn{k}\brk{T_{V_{k}}}$.  From Corollary
  \ref{C:LaurentPoly}, $\vp_e$ is $2r$-periodic.  Furthermore, 
  Proposition \ref{P:SandV} implies that for $k\in\{1,\ldots,r-1\}$,
  one has
  $$\vp_e(k)+\vp_e(-k)=(q^{k}-q^e-q^{-k}+q^e)\qn{k}\brk{T_{S_{r-1-k}}}
  =\qn{k}^2\brk{T_{S_{r-1-k}}}.$$
  So, using that $\vp_0(0)=\vp_e(r)=0$, we can write
  \begin{align*}
    \Nr^0(M,\emptyset,\omega)
    &=c\sum_{\tiny{
        \begin{array}{c}
          k\in e+2\Z\\0< k<2r
        \end{array}}}\vp_e(k)
    =c\bp{\sum_{\tiny{
        \begin{array}{c}
          k\in e+2\Z\\0< k<r
        \end{array}}}\vp_e(k)+\sum_{\tiny{
        \begin{array}{c}
          k\in e+2\Z\\-r< k<0
        \end{array}}}\vp_e(k)}\\
    &=c\sum_{\tiny{
        \begin{array}{c}
          k\in e+2\Z\\0< k<r
        \end{array}}}\qn{k}^2\brk{T_{S_{r-1-k}}}
    =\frac{rf}{\qn1\Delta_{\sign(f)}}\sum_{\tiny{
        \begin{array}{c}
          k\in e+2\Z\\0< k<r
        \end{array}}}\qn{r-k}^2\brk{T_{S_{r-1-k}}}\\
    &=\frac{rf}{\qn1\Delta_{\sign(f)}}
    \sum_{\tiny{
        \begin{array}{c}
          n\in\omega+2\Z\\0\le n\le r-2
        \end{array}}}\qn{n+1}^2\brk{T_{S_n}}.
  \end{align*}
  Finally, $\qdim(S_n)=(-1)^n\qN{n+1}$ implies 
$$\Nr^0(M,\emptyset,\omega)=\frac{|f|}{\Delta^\sot_{\sign(f)}}
\sum_{\tiny{
        \begin{array}{c}
          n\in\omega+2\Z\\0\le n\le r-2
        \end{array}}}\qdim(S_n)J_n(K)= |f|WRT(M,\emptyset,\omega).$$
    \renewcommand{\qedsymbol}{\fbox{\ref{ThMknot}}}
\end{proof}
\renewcommand{\qedsymbol}{\fbox{\theDf}}

\section{Vanishing of $\Nr^0$ for non-homology spheres}\label{S:ConjNonHomol}

\begin{theorem}
  Let $(M,T,\coh)$ any compatible triple.  If $b_1(M)>0$ then $\Nr^0(M,T,\coh)=0$.
\end{theorem}
\begin{proof}
  Since $b_1(M)>0$ there exists a non-trivial $\delta \in H^1(M;
  \Z)\subset H^1(M; \C)$.  For $\alpha\in \C$ let $\bar{\alpha}\delta
  \in H^1(M\setminus T; \C/2\Z)$ be the trivial extension of
  $\bar{\alpha}\delta \in H^1(M; \C/2\Z)$ where $\bar{\alpha}$ is the
  image of $\alpha$ in $\C/2\Z$.
  Then $(M,T,\coh+\bar{\alpha}\delta)$ is a compatible triple for all
  $\alpha\in \C$.  Moreover, there exists a neighborhood $N$ of $0\in
  \C$ such that $\coh +\bar \alpha \delta$ is non-integral for all
  $\alpha \in N\setminus \{0\}$.  Then for a complex number $\alpha
  \in N\setminus \{0\}$, 
  Propositions 1.5 and 3.14 of \cite{CGP} implies that
  $\Nr^0(M,T,\coh+\bar \alpha \delta)=0$.
 
  Now, for all $\alpha\in \C$, by definition of $\Nr^0$ we have
 $$\Nr^0(M,T,\coh+\bar \alpha \delta)=\Nr((M,T,\coh+\bar \alpha \delta)
 \#(S^3,o_\beta,\coh_\beta))/ \qd(\beta)$$ where $o_\beta$ is the
 unknot in $S^3$ colored by $V_\beta$, $\beta\in \srcol$ and
 $\coh_\beta$ be the unique element of $H^1(S^3\setminus
 o_\beta,\C/2\Z)$ such that $(S^3,o_\beta,\coh_\beta)$ is a compatible
 triple.  
 To compute the right side of this equation, we choose a 
 link $L^{\omega\#\omega_\beta}\cup T\cup o_\beta$  which is a  computable presentation 
 of $(M,T,\coh) \#(S^3,o_\beta,\coh_\beta)$.  Then the same link colored
 by $\omega_\alpha'=(\omega+\bar \alpha \delta)\#\omega_\beta$ gives a
 presentation of 
 $(M,T,\coh+\bar \alpha \delta) \#(S^3,o_\beta,\coh_\beta)$.  
For each component $L_i$ of $L^{\omega'_\alpha}$ the color $g_{\omega_\alpha'}(L_i)$ is an affine functions of $\alpha$.  The link $L^{\omega'_\alpha}$ is computable if and only if all the  colors  $g_{\omega_\alpha'}(L_i)$ are in  $\C/2\Z\setminus\Z/2\Z$.  Let $N'$ be the open set of $\C$ consisting of $\alpha$ such that $L^{\omega'_\alpha}$ is computable.  Then $N'$ contains $0$ since $L^{\omega'_0}$ is computable.  
 
 Now we have $$\Nr^0(M,T,\coh+\bar \alpha \delta)=\frac{F'(L^{\omega'_\alpha} \cup T\cup o_\beta)}{\qd(\beta)\Delta_+^{p}\ \Delta_-^{s}}=\frac{\brk{L^{\omega'_\alpha} \cup T\cup |_{V_\beta}}}{\Delta_+^{p}\ \Delta_-^{s}}$$
 where $ |_{V_\beta}$ is the trivial one-component (1-1)-tangle
 colored with $V_\beta$.  
 The function 
 $$\alpha \mapsto \brk{L^{\omega'_\alpha} \cup T\cup
   |_{V_\beta}}$$ is continuous on $N'$ because it is a weighted sum of
 continuous functions (by Corollary \ref{C:TangleHolom}) where the
 weights are products of functions $\qd$ evaluated 
 away from their
 poles.  Thus, 
 $\Nr^0(M,T,\coh+\bar \alpha \delta)$
 is continuous at
 $\alpha=0$.  
Finally, since $\Nr^0(M,T,\coh+\bar \alpha \delta)$ vanishes on $N$, we have $\Nr^0(M,T,\coh)=~0$.
\end{proof}

\end{document}